\documentclass[12pt,thmsa]{article}
\usepackage{amssymb}
\usepackage{amsmath}
\usepackage{amsfonts}
\newtheorem{theorem}{Theorem}

\newtheorem{corollary}[theorem]{Corollary}

\newtheorem{definition}[theorem]{Definition}
\newtheorem{example}[theorem]{Example}

\newtheorem{proposition}[theorem]{Proposition}
\newtheorem{remark}[theorem]{Remark}

\newenvironment{proof}[1][Proof]{\noindent \textbf{#1.} }{\  \rule{0.5em}{0.5em}}

\begin{document}

\title{\textbf{Normal Transport Surfaces in Euclidean 4-space $\mathbb{E}%
^{4} $ }}
\author{K. Arslan, B. Bulca, B. K. Bayram \& G. \"{O}zt\"{u}rk }
\date{December 2014}
\maketitle

\begin{abstract}
In the present paper we study normal transport surfaces in four-dimensional
Euclidean space $\mathbb{E}^{4}$ which are the generalization of surface
offsets in $\mathbb{E}^{3}$. We find some results of normal transport
surfaces in $\mathbb{E}^{4}$ of evolute and parallel type. Further, we give
some examples of these type of surfaces.
\end{abstract}

\section{Introduction}

\footnote{%
2000 Mathematics Subject Classifications: 53A04, 53C42
\par
Key words and phrases: Translation surface, parallel surface, evolute
surface, focal surface
\par
{}}The geometric modelling of free-form curves and surfaces is of central
importance for sophisticated CAD/CAM systems. Apart from the pure
construction of these curves and surfaces, the analysis of their quality is
equally important in the design and manufacturing process. It is for example
very important to test the convexity of a surface, to pinpoint inflection
points, to visualize flat points and to visualize technical smoothness of
surface \cite{HH1}.

The 3D offsets or parallel surfaces are very widely used in many
applications. these include tool path generation for 3N machining. \ However
3D offsets are particularly important and useful as pre-process
modifications to CAD geometry. By defining an 3D offset means moving a
surface of a 3D model by a constant "d" in a direction normal to the surface
of the model. Offset techniques for surfaces has been extensively studied by
Mechawa (\cite{Ma}) and Pham (\cite{P}). Generally offsets of 3D models are
achieved by first offsetting all surfaces \ of the model and then trimming
and extending these offsets to reconstruct a closed 3D model (\cite{Fa},
\cite{Fo}).

Further, focal surfaces are known in the field of line congruences. Line
congruences have been introduced in the field of visualization by Hagen and
Pottmann (see, \cite{Hha}). Focal surfaces are also used as a surface
interrogation tool to analyses the "quality" of the surface before further
processing of the surface, for example in a NC-milling operation (see \cite%
{HH1}).

The generalized focal surfaces are related to hedgehog diagrams. Instead of
drawing surface normals proportional to a surface value, only the point on
the surface normal proportional to the function is drawing. The loci of all
these points is the generalized focal surface. This method was introduced by
Hagen and Hahmann (\cite{HH1}, \cite{HH2}) and is based on the concept of
focal surface which are known from line geometry. The focal surfaces are the
loci of all focal points of special congruence, the normal congruence.
Recently the present authors considered parallel and focal surfaces and
their curvature properties (see \cite{OA}).

The normal transport surface $\widetilde{M}$ of $M$ are generalization of
offset surfaces to $4$-dimensional Euclidean space $\mathbb{E}^{4}$ \cite{Fr}%
. Observe that, evolute surfaces and parallel type surfaces in $\mathbb{E}%
^{4}$ are the special type normal transport surfaces \cite{Kr}, \cite{Ce},
\cite{Fr}. Parallel type surface are widely used in geometry and
mathematical physics. We want to refer the reader to da Costa \cite{Co} for
an application in quantum mechanics in curved spaces.

The paper organized as follows. In section $2,$ we briefly considered basic
concepts of surfaces in Euclidean spaces. In section $3,$ we consider some
known results about the surfaces with flat normal bundle. In the final
section, we consider normal transport surfaces in $\mathbb{E}^{4}.$ Further
we give some examples of evolute and parallel type surfaces in $\mathbb{E}%
^{4}$.

\section{\textbf{Preliminaries}}

In the present section we recall definitions and results of \cite{Fr}. Let $%
M $ be a local surface in $\mathbb{E}^{n+2}$ given with the regular patch $%
x(u,v)$ : $(u,v)\in D\subset \mathbb{E}^{2}$. The tangent space $T_{p}(M)$
to $M$ at an arbitrary point $p=x(u,v)$ of $M$ is spanned by $\left \{
x_{u},x_{v}\right \} $. Further, the coefficients of the first fundamental
form of $M$ are given by
\begin{equation}
g_{_{11}}=\left \langle x_{u},x_{u}\right \rangle ,g_{_{12}}=\left \langle
x_{u},x_{v}\right \rangle ,g_{22}=\left \langle x_{v},x_{v}\right \rangle ,
\label{2.1}
\end{equation}%
where $\left \langle ,\right \rangle $ is the Euclidean inner product. Let
us denote by%
\begin{equation}
ds^{2}=\sum \limits_{i,j=1}^{2}g_{ij}du^{i}du^{j}.  \label{2.2}
\end{equation}

Let us choose $n$ linearly independent, orthogonal unit normal vectors $%
N_{\alpha },$ $\alpha =1,2,...,n$ spanning the normal space $T_{p}^{\perp }M$
at point $p=x(u,v).$ For each $p\in M$, consider the decomposition $T_{p}%
\mathbb{E}^{n+2}=T_{p}M\oplus T_{p}^{\perp }M,$ where $T_{p}^{\perp }M$ is
the orthogonal component of $T_{p}M$ in $\mathbb{E}^{n+2}$. Let $\overset{%
\sim }{\nabla }$ be the Riemannian connection of $\mathbb{E}^{n+2}$ then the
Gauss equation of the surface $M$ is given by
\begin{equation}
x_{u^{i}u^{j}}=\widetilde{\nabla }_{x_{u^{i}}}x_{u^{j}}=\sum_{k=1}^{2}\Gamma
_{ij}^{k}x_{u^{k}}+\sum_{\alpha =1}^{n}\ c_{ij}^{\alpha }N_{\alpha },
\label{2.3}
\end{equation}%
where
\begin{equation}
c_{ij}^{\alpha }=\left\langle x_{u^{i}u^{j}},N_{\alpha }\right\rangle ;\text{
}c_{ij}^{\alpha }=c_{ji}^{\alpha },  \label{2.4}
\end{equation}%
are the coefficients of the second fundamental form and%
\begin{equation}
\Gamma _{ij}^{k}=\sum_{l=1}^{2}g^{lk}\left( \frac{\partial g_{jl}}{\partial
u^{i}}+\frac{\partial g_{li}}{\partial u^{j}}-\frac{\partial g_{ij}}{%
\partial u^{l}}\right) ,  \label{2.5}
\end{equation}%
are the Christoffel symbols corresponding to $x(u,v)$.

Further, the Weingarten equation of the surface $M$ is given by%
\begin{equation}
(N_{\alpha })_{u^{i}}=\widetilde{\nabla }_{x_{u^{i}}}N_{\alpha
}=-\sum_{k=1}^{2}\ c_{\alpha }^{ik}x_{u^{k}}+\sum_{\beta =1}^{n}\
T_{i}^{\alpha \beta }N_{\beta },  \label{2.6}
\end{equation}%
where%
\begin{equation}
c_{\alpha }^{ik}=\sum_{j=1}^{2}\ c_{ij}^{\alpha }g^{jk};\text{ }c_{\alpha
}^{ik}=c_{\alpha }^{ki},  \label{2.7}
\end{equation}%
are the Weingarten forms of M with respect to some unit normal vector $%
N_{\alpha }$ and \
\begin{equation}
T_{i}^{\alpha \beta }=\left \langle (N_{\alpha })_{u^{i}},N_{\beta }\right
\rangle ;T_{i}^{\alpha \beta }=-T_{i}^{\beta \alpha }\text{, }i=1,2,
\label{2.8}
\end{equation}%
are the torsion coefficients\ with $\alpha ,\beta =1,...,n$ and
\begin{equation}
\left( g^{ij}\right) _{i,j=1,2}=\frac{1}{g}\left[
\begin{array}{cc}
g_{22} & -g_{12} \\
-g_{21} & g_{11}%
\end{array}%
\right] ,\text{ }g=g_{11}g_{22}-g_{12}^{2}\text{.}  \label{2.9}
\end{equation}%
A simple calculation shows that%
\begin{equation}
c_{ij}^{\alpha }=\left \langle x_{u^{i}u^{j}},N_{\alpha }\right \rangle
=-\left \langle x_{u^{i}},\left( N_{\alpha }\right) _{u^{j}}\right \rangle .
\label{2.9*}
\end{equation}

The Gaussian curvature of the surface $M$ is defined by%
\begin{equation}
K=\sum_{\alpha =1}^{n}K_{\alpha },\text{ \ \ }K_{\alpha }=\frac{%
c_{11}^{\alpha }c_{22}^{\alpha }-(c_{12}^{\alpha })^{2}}{g}.  \label{2.10}
\end{equation}%
where $K_{\alpha }$ is the $\alpha ^{th}$ Gaussian curvature of the surface $%
M$. The Gaussian curvature vanishes identically for so-called flat surface.
Observe that
\begin{equation}
\ K_{\alpha }=c_{\alpha }^{11}c_{\alpha }^{22}-(c_{\alpha }^{12})^{2}.
\label{2.11}
\end{equation}

The mean curvature vector field $\overrightarrow{H}$ of the surface $M$ is
defined by%
\begin{equation}
\overrightarrow{H}=\sum_{\alpha =1}^{n}H_{\alpha }N_{\alpha },  \label{2.12}
\end{equation}%
where
\begin{equation}
H_{\alpha }=\frac{1}{2}\sum \limits_{i,j=1}^{2}\text{\ }g^{ij}c_{ij}^{%
\alpha }=\frac{g_{22}c_{11}^{\alpha }+g_{11}c_{22}^{\alpha
}-2g_{12}c_{12}^{\alpha }}{2g},  \label{2.13}
\end{equation}%
is the $\alpha ^{th}$\textbf{\ }mean curvature of the surface $M$ with
respect to the unit normal vector $N_{\alpha }.$ The mean curvature $H$ of $%
M $ is defined by $H=\left \Vert \overrightarrow{H}\right \Vert .~$The mean
curvature (vector) vanishes identically for so-called minimal surface.
Observe that
\begin{equation}
H_{\alpha }=\frac{c_{\alpha }^{11}+c_{\alpha }^{22}}{2}.  \label{2.14}
\end{equation}

The curvature tensor of the normal bundle $NM$ of the surface $M$ is defined
by%
\begin{equation}
\begin{array}{l}
S_{ij}^{\alpha \beta }=\left( T_{i}^{\alpha \beta }\right)
_{u^{j}}-\left( T_{j}^{\alpha \beta }\right)
_{u^{i}}+\sum\limits_{\sigma =1}^{n}\left( T_{i}^{\alpha \sigma
}T_{j}^{\sigma \beta }-T_{j}^{\alpha \sigma
}T_{i}^{\sigma \beta }\right) , \\
=\sum\limits_{m,n=1}^{2}\left( c_{1m}^{\alpha }c_{n2}^{\beta
}-c_{2m}^{\alpha }c_{n1}^{\beta }\right) \text{\ }g^{mn};1\leq
\alpha ,\beta
\leq n.%
\end{array}
\label{2.15*}
\end{equation}%
The equality
\begin{equation}
S_{N}^{\alpha \beta }=\frac{1}{\sqrt{g}}S_{12}^{\alpha \beta },  \label{2.16}
\end{equation}%
is called the normal sectional curvature with respect to the plane $\Pi
=span\left\{ x_{u},x_{v}\right\} $. For the case $n=2$ the scalar curvature
of its normal bundle is defined as%
\begin{equation}
K_{N}=S_{N}^{12}=\frac{1}{\sqrt{g}}S_{12}^{12}.  \label{2.17}
\end{equation}%
which is also called normal curvature of the surface $M$ in $\mathbb{E}^{4}.$
Observe that%
\begin{equation}
K_{N}=\frac{1}{\sqrt{g}}\left( \left( T_{2}^{12}\right) _{u}-\left(
T_{1}^{12}\right) _{v}\right) .  \label{2.18}
\end{equation}

\section{\textbf{Known Results}}

Let M be a local surface in $\mathbb{E}^{n+2}$ given with the surface patch $%
x(u,v)$ : $(u,v)\in D\subset \mathbb{E}^{2}.$ The mean curvature vector $%
\overrightarrow{H}$ is called parallel in the normal bundle if and only if
\begin{equation}
\left( H_{\alpha }\right) _{u}^{\bot }=0,\left( H_{\alpha }\right)
_{v}^{\bot }=0,  \label{4.1}
\end{equation}%
or equivalently%
\begin{equation}
\left( H_{\alpha }\right) _{u^{i}}=\sum \limits_{\beta
=1}^{n}H_{\beta }T_{i}^{\alpha \beta }.  \label{4.2}
\end{equation}%
for all $i=1,2$, $\alpha =1,...,n$ with respect to an arbitrary orthonormal
frame $N_{1},...N_{n}~$\cite{Fr}.

\begin{proposition}
\cite{Fr} The mean curvature vector $\overrightarrow{H}$ is called parallel
in the normal bundle if and only if the squared mean curvature $\left \Vert
\overrightarrow{H}\right \Vert ^{2}$of $M$ is a constant function.
\end{proposition}

\begin{proof}
Suppose $H_{\alpha }\neq 0$ and the mean curvature vector $\overrightarrow{H}
$ is parallel in the normal bundle. Multiplying the first order differential
equation (\ref{4.2}) by $H_{\alpha }$ gives
\begin{equation*}
H_{\alpha }\left( H_{\alpha }\right) _{u^{i}}=\sum \limits_{\beta
=1}^{n}H_{\alpha }H_{\beta }T_{i}^{\alpha \beta }=0,
\end{equation*}%
for all $\alpha =1,...,n.$ Summing over $\alpha $ shows%
\begin{eqnarray*}
\frac{1}{2}\frac{\partial }{\partial u^{i}}\left \Vert \overrightarrow{H}%
\right \Vert ^{2} &=&\sum \limits_{\alpha =1}^{n}H_{\alpha }\left(
H_{\alpha }\right) _{u^{i}} \\
&=&\sum \limits_{\alpha =1}^{n}\sum \limits_{\beta =1}^{n}H_{\alpha
}H_{\beta }T_{i}^{\alpha \beta }=0,
\end{eqnarray*}%
where the right hand side vanishes automatically due to the skew-symmetric
of the torsion coefficients \cite{Fr}. Thus, one get
\begin{equation*}
\left \Vert \overrightarrow{H}\right \Vert ^{2}=\sum \limits_{\alpha
=1}^{n}H_{\alpha }^{2}=const.
\end{equation*}%
The converse statement of the theorem is trivial.
\end{proof}

\begin{definition}
A local surface of \ $\mathbb{E}^{n+2}$ is said to have flat normal bundle
if and only if the orthonormal frame $N_{1},...N_{n}$ of $M$ is of torsion
free.
\end{definition}

\begin{remark}
The existence of flat normal bundle of $M$ is equivalent to say that normal
curvature $K_{N}$ of $M$ vanishes identically.
\end{remark}

The following classification result due to Chen from \cite{Ch}.

\begin{theorem}
Let $M$ be an immersed surface in $\mathbb{E}^{n+2}$. If $\overrightarrow{H}%
\neq 0$ is parallel in the normal bundle then either $M$ is a minimal
surface of a hypersphere of $\mathbb{E}^{n+2}$, or it has flat normal bundle.
\end{theorem}

\section{\textbf{Normal Transport Surfaces in }$\mathbb{E}^{4}$}

Let $M$ and $\widetilde{M}$ be two smooth surfaces in Euclidean $4$-space $%
\mathbb{E}^{4}$ and let $\varphi :M\rightarrow \widetilde{M}$ be a
diffeomorphism. Then the surface $\widetilde{M}$ enveloping family of normal
$2$-planes to $M$ is the normal transport of $M$ in $\mathbb{E}^{4}$ \cite%
{Fr}. Furthermore, let $\overrightarrow{x}$ be a position (radius) vector of
$p\in M,$ and $\widetilde{x}$ be the position (radius) vector of the point $%
\varphi (p)\in \widetilde{M}.$ Then the mapping $\varphi :M\rightarrow
\widetilde{M}$ has the form%
\begin{equation*}
\widetilde{x}=x+\overrightarrow{w},~~\overrightarrow{w}\in T_{p}^{\perp }M.
\end{equation*}%
where, $\overrightarrow{p\varphi (p)}=\overrightarrow{w}(p),~\overrightarrow{%
w}(p)\in T_{p}^{\perp }M$ is the normal vector to $M$. For the case%
\begin{equation*}
\overrightarrow{w}(p)=\sum \limits_{i=1}^{2}f_{i}(u,v)N_{i}(u,v),
\end{equation*}%
the normal transport surface $\widetilde{M}$ of $M$ given by
\begin{equation}
\widetilde{M}:\widetilde{x}(u,v)=x(u,v)+\sum%
\limits_{i=1}^{2}f_{i}(u,v)N_{i}(u,v),  \label{5.1}
\end{equation}%
where $f_{i}$ \ $(i=1,2)$ are offset functions and $N_{1},N_{2}\in
T_{p}^{\perp }M$ \cite{Fr}.

The tangent space to $\widetilde{M}$ at an arbitrary point $p=\widetilde{x}%
(u,v)$ of $\widetilde{M}$ is spanned by%
\begin{equation}
\begin{array}{l}
\vspace{2mm}\widetilde{x}_{u}=x_{u}+f_{1}\left( N_{1}\right)
_{u}+f_{2}\left( N_{2}\right) _{u}+(f_{1})_{u}N_{1}+(f_{2})_{u}N_{2}, \\
\widetilde{x}_{v}=x_{v}+f_{1}\left( N_{1}\right) _{v}+f_{2}\left(
N_{2}\right) _{v}+(f_{1})_{v}N_{1}+(f_{2})_{v}N_{2}.%
\end{array}
\label{5.2}
\end{equation}%
Further, using the Weingarten equation (\ref{2.6}) we get%
\begin{equation}
\begin{array}{l}
\left( N_{1}\right) _{u}=-\left( c_{1}^{11}x_{u}+c_{1}^{12}x_{v}\right)
+T_{1}^{12}N_{2} \\
\left( N_{2}\right) _{u}=-\left( c_{2}^{11}x_{u}+c_{2}^{12}x_{v}\right)
-T_{1}^{12}N_{1} \\
\left( N_{1}\right) _{v}=-\left( c_{1}^{21}x_{u}+c_{1}^{22}x_{v}\right)
+T_{2}^{12}N_{2} \\
\left( N_{2}\right) _{v}=-\left( c_{2}^{21}x_{u}+c_{2}^{22}x_{v}\right)
-T_{2}^{12}N_{2}.%
\end{array}
\label{5.3}
\end{equation}%
So, substituting (\ref{5.3}) into (\ref{5.2}) we get%
\begin{equation}
\begin{array}{l}
\widetilde{x}_{u}=\left( 1-f_{1}c_{1}^{11}-f_{2}c_{2}^{11}\right)
x_{u}-\left( f_{1}c_{1}^{12}+f_{2}c_{2}^{12}\right) x_{v} \\
\text{ \ \ \ \ \ }+\left( (f_{1})_{u}-f_{2}T_{1}^{12}\right) N_{1}+\left(
(f_{2})_{u}+f_{1}T_{1}^{12}\right) N_{2},%
\end{array}
\label{5.4}
\end{equation}%
\begin{equation}
\begin{array}{l}
\widetilde{x}_{v}=-\left( f_{1}c_{1}^{21}+f_{2}c_{2}^{21}\right)
x_{u}+\left( 1-f_{1}c_{1}^{22}-f_{2}c_{2}^{22}\right) x_{v} \\
\text{ \ \ \ \ \ }+\left( (f_{1})_{v}-f_{2}T_{2}^{12}\right) N_{1}+\left(
(f_{2})_{v}+f_{1}T_{2}^{12}\right) N_{2}.%
\end{array}
\label{5.4*}
\end{equation}

The normal transport surfaces in $3$-dimensional Euclidean space $\mathbb{E}%
^{3}$ have the parametrization
\begin{equation*}
\widetilde{M}:\widetilde{x}(u,v)=x(u,v)+F(u,v)~N(u,v),
\end{equation*}%
where $N(u,v)\in T_{p}^{\perp }M$ and $F$ is a real valued function in the
parameter $(u,v)$. In fact, these surfaces are known as surface offsets in $%
\mathbb{E}^{3}~$and $F$ is its offset function \cite{H}.

If the offset function depends on the principal curvatures $k_{1}$ and $%
k_{2} $ of $M$ then one can choose the variable offset function as;

\begin{enumerate}
\item $F=k_{1}k_{2},~$Gaussian curvature,

\item $F=\frac{1}{2}(k_{1}+k_{2}),$ mean curvature,

\item $F=k_{1}^{2}+k_{2}^{2},$ energy functional,

\item $F=\left \vert k_{1}\right \vert +\left \vert k_{2}\right \vert ,$
absolute functional,

\item $F=k_{i},1\leq i\leq 2,$ principal curvature,

\item $F=\frac{1}{k_{i}},$ focal points,

\item $F=const.,$ parallel surface.
\end{enumerate}

The different offset functions listed above can now be used the interrogate
and visualize of the surfaces (see \cite{OA}). \ Using different offset
functions, one can construct a one-parameter family of various normal
transport surfaces from a given surface of $4$-dimensional Euclidean space $%
\mathbb{E}^{4}$.

In the following definition we construct some special normal transport
surfaces in $\mathbb{E}^{4}$ which are the generalization of some
generalized focal surfaces give before$.$

\begin{definition}
i) The normal transport surface $\widetilde{M}_{H}$ given with the
parametrization
\begin{equation}
\widetilde{M}_{H}:\widetilde{x}%
(u,v)=x(u,v)+H_{1}(u,v)~N_{1}(u,v)+H_{2}(u,v)~N_{2}(u,v),  \label{5.8}
\end{equation}%
is called normal transport surface of $H$-type, where $f_{\alpha
}(u,v)=H_{\alpha }$ $(\alpha =1,2)$ are the offset functions.

ii) The normal transport surface $\widetilde{M}_{K}$ given with the
parametrization
\begin{equation}
\widetilde{M}_{K}:\widetilde{x}%
(u,v)=x(u,v)+K_{1}(u,v)~N_{1}(u,v)+K_{2}(u,v)~N_{2}(u,v),  \label{5.9}
\end{equation}%
is called normal transport surface of $K$-type, where $f_{\alpha
}(u,v)=K_{\alpha }$ $(\alpha =1,2)$ are the offset functions.
\end{definition}

\subsection{Parallel Surfaces in $\mathbb{E}^{4}$}

\begin{definition}
The normal transport surface $\widetilde{M}$ of $M$ is called parallel
surface of $M$ in $\mathbb{E}^{4}$ if the equality
\begin{equation}
\left \langle \widetilde{x}_{u_{i}},N_{\alpha }\right \rangle =0,\text{ }%
1\leq i,\alpha \leq 2,  \label{5.5}
\end{equation}%
holds for all $N_{\alpha }\in T_{p}^{\perp }M$ \cite{Fr}.
\end{definition}

If the functions $f_{1}$ and $f_{2}$ are constant then it is easy to see
that $\widetilde{M}$ is a parallel surface of $M$ and vice versa, at least
if the surfaces are immersed in $\mathbb{E}^{3}.$ The parallelity of $%
\widetilde{M}$ in $\mathbb{E}^{4}$ depends on the normal curvature $K_{N}$
of $M$ \cite{Fr}. Parallel type surface are widely used in geometry and
mathematical physics. We want to refer the reader to da Costa \cite{Co} for
an application in quantum mechanics in curved spaces.

Let $\widetilde{M}$ be a parallel surface of $M$ in $\mathbb{E}^{4}$. Then
by use of (\ref{5.4}) and (\ref{5.4*}) with (\ref{5.5}) one can get%
\begin{eqnarray}
0 &=&\left \langle \widetilde{x}_{u},N_{1}\right \rangle
=(f_{1})_{u}-f_{2}T_{1}^{12},  \notag \\
0 &=&\left \langle \widetilde{x}_{v},N_{1}\right \rangle
=(f_{1})_{v}-f_{2}T_{2}^{12},  \label{5.6} \\
0 &=&\left \langle \widetilde{x}_{u},N_{2}\right \rangle
=(f_{2})_{u}+f_{1}T_{1}^{12},  \notag \\
0 &=&\left \langle \widetilde{x}_{v},N_{2}\right \rangle
=(f_{2})_{v}+f_{1}T_{2}^{12}.  \notag
\end{eqnarray}%
Differentiating the first two equations and making use of the other
equations shows us
\begin{eqnarray}
(f_{1})_{uv}+f_{1}T_{2}^{12}T_{1}^{12}-f_{2}\left( T_{1}^{12}\right) _{v}
&=&0,  \label{5.7} \\
(f_{1})_{vu}+f_{1}T_{1}^{12}T_{2}^{12}-f_{2}\left( T_{2}^{12}\right) _{u}
&=&0.  \notag
\end{eqnarray}%
Thus a computation of the left hand sides of (\ref{5.7}) brings%
\begin{equation*}
-f_{2}\left \{ \left( T_{1}^{12}\right) _{v}-\left( T_{2}^{12}\right)
_{u}\right \} =0.
\end{equation*}%
So, by the use of (\ref{2.18}) we can conclude that the normal curvature $%
K_{N}$ of $M$ vanishes identically \cite{Fr}. Consequently, we obtain the
following result of S. Fr\"{o}hlich.

\begin{theorem}
\cite{Fr} The normal transport surface $\widetilde{M}$ of $M$ is parallel if
and only if $M$ has flat normal bundle.
\end{theorem}

We obtain the following result.

\begin{corollary}
The normal transport surface $\widetilde{M}$ of $M$ is parallel if
and only if the squared sum of the offset functions is constant,
i.e., $\sum \limits_{i=1}^{2}f_{i}^{2}(u,v)=const.$
\end{corollary}

\begin{proof}
From the expressions in (\ref{5.6}) we get%
\begin{equation}
\begin{array}{l}
(f_{1})_{u}f_{1}+(f_{2})_{u}f_{2}=0, \\
(f_{1})_{v}f_{1}+(f_{2})_{v}f_{2}=0.%
\end{array}
\label{5.7.*}
\end{equation}%
which completes the proof.
\end{proof}

We give the following example.

\begin{example}
The normal transport surface $\widetilde{M}$ of $M$ is given with the patch%
\begin{equation*}
\widetilde{X}(u,v)=X(u,v)+r\cos u~N_{1}(u,v)+r\sin u~N_{2}(u,v),
\end{equation*}%
is a parallel surface of $M$ in $\mathbb{E}^{4}.$
\end{example}

Let $M$ be a non-minimal local surface in $\mathbb{E}^{4}$ and $\widetilde{M}%
_{H}$ its normal transport surface. If $\widetilde{M}_{H}$ is a parallel
surface of $M$ in $\mathbb{E}^{4}$ then by Theorem 3 $M$ has vanishing
normal curvature. Furthermore, by the use of (\ref{5.7.*}) we get
\begin{eqnarray*}
\left( H_{1}\right) _{u}H_{1}+\left( H_{2}\right) _{u}H_{2} &=&0, \\
\left( H_{1}\right) _{v}H_{1}+\left( H_{2}\right) _{v}H_{2} &=&0.
\end{eqnarray*}%
Thus, $\left \Vert \overrightarrow{H}\right \Vert ^{2}=\sum
\limits_{\alpha =1}^{2}H_{\alpha }^{2}$ is a constant function. So,
we conclude that the mean curvature vector $\overrightarrow{H}$ of
$M$ is parallel in the normal bundle. Thus, we have proved the
following result.

\begin{theorem}
Let $M$ be a non-minimal local surface in $\mathbb{E}^{4}$. Then the normal
transport surface $\widetilde{M}_{H}$ of $M$ in $\mathbb{E}^{4}$ is parallel
if and only if the mean curvature vector $\overrightarrow{H}$ of $M$ is
parallel in the normal bundle.
\end{theorem}

Let $M$ be a non-flat local surface in $\mathbb{E}^{4}$ and $\widetilde{M}%
_{K}$ its normal transport surface. If $\widetilde{M}_{K}$ is a parallel
surface of $M$ in $\mathbb{E}^{4}$ then by Theorem 3 $\widetilde{M}_{K}$ has
vanishing normal curvature. Furthermore, by the use of (\ref{5.7.*}) we get
\begin{eqnarray*}
\left( K_{1}\right) _{u}K_{1}+\left( K_{2}\right) _{u}K_{2} &=&0, \\
\left( K_{1}\right) _{v}K_{1}+\left( K_{2}\right) _{v}K_{2} &=&0.
\end{eqnarray*}%
Thus, we conclude that $K=\sum \limits_{\alpha =1}^{2}K_{\alpha
}^{2}$ is a constant function, i.e., $M$ has constant Gauss
curvature. Thus, we have proved the following result.

\begin{theorem}
Let $M$ be a non-flat local surface in $\mathbb{E}^{4}$. Then the normal
transport surface $\widetilde{M}_{K}$ of $M$ in $\mathbb{E}^{4}$ is parallel
if and only if the Gaussian curvature of $M$ is a non-zero constant.
\end{theorem}

\subsection{Evolute Surfaces in $\mathbb{E}^{4}$}

\begin{definition}
The normal transport surface $\widetilde{M}$ of $M$ is called evolute
surface of $M$ in $\mathbb{E}^{4}$ if the equality
\begin{equation}
\left \langle \widetilde{x}_{u_{i}},x_{u_{j}}\right \rangle =0,\text{ }1\leq
i,j\leq 2,  \label{6.1}
\end{equation}%
holds for all $x_{u_{j}}\in T_{p}M$ .
\end{definition}

Observe that, The tangent $2$-planes at a point $p\in M$ and at the
corresponding point $\varphi (p)\in \widetilde{M}$ are mutually
orthogonal, and the vector $\overrightarrow{p\varphi
(p)}=~\overrightarrow{w}(p)$, $\overrightarrow{w}(p)\in T_{p}^{\perp }M$ is the normal vector to $M$ \cite%
{Ce}.

Let $\widetilde{M}$ be a evolute surface of $M$ in $\mathbb{E}^{4}$. Then by
use of (\ref{5.4}) with (\ref{6.1}) we can get%
\begin{eqnarray}
0 &=&\left \langle \widetilde{x}_{u},x_{u}\right \rangle =\left(
1-f_{1}c_{1}^{11}-f_{2}c_{2}^{11}\right) g_{11}-\left(
f_{1}c_{1}^{12}+f_{2}c_{2}^{12}\right) g_{21},  \notag \\
0 &=&\left \langle \widetilde{x}_{u},x_{v}\right \rangle =\left(
1-f_{1}c_{1}^{11}-f_{2}c_{2}^{11}\right) g_{12}-\left(
f_{1}c_{1}^{12}+f_{2}c_{2}^{12}\right) g_{22},  \label{6.2} \\
0 &=&\left \langle \widetilde{x}_{v},x_{u}\right \rangle =-\left(
f_{1}c_{1}^{12}+f_{2}c_{2}^{12}\right) g_{11}+\left(
1-f_{1}c_{1}^{22}-f_{2}c_{2}^{22}\right) g_{21},  \notag \\
0 &=&\left \langle \widetilde{x}_{v},x_{v}\right \rangle =-\left(
f_{1}c_{1}^{12}+f_{2}c_{2}^{12}\right) g_{12}+\left(
1-f_{1}c_{1}^{22}-f_{2}c_{2}^{22}\right) g_{22}.  \notag
\end{eqnarray}%
From now on we assume that the surface patch $x(u,v)$ satisfies the metric
condition $g_{12}=0.$ So the equations in (\ref{6.2}) turn into%
\begin{eqnarray}
f_{1}c_{1}^{11}+f_{2}c_{2}^{11} &=&1,  \notag \\
f_{1}c_{1}^{22}+f_{2}c_{2}^{22} &=&1,  \label{6.3} \\
f_{1}c_{1}^{12}+f_{2}c_{2}^{12} &=&0.  \notag
\end{eqnarray}%
Consequently by the use of (\ref{6.3}) with (\ref{2.14}) we get
\begin{equation}
f_{1}H_{1}+f_{2}H_{2}=1.  \label{6.4}
\end{equation}%
So, we obtain the following result.

\begin{theorem}
Let $M$ be local surface in $\mathbb{E}^{4}$ with $g_{12}=0$. Then the
normal transport surface $\widetilde{M}$ in $\mathbb{E}^{4}$ is evolute
surface of $M~$if and only if the first and second mean curvatures $%
H_{1},~H_{2}~$satisfies the condition (\ref{6.4}).
\end{theorem}

\begin{corollary}
Let $M$ be local surface in $\mathbb{E}^{4}$ with $g_{12}=0$. Then the
normal transport surface $\widetilde{M}_{H}$ in $\mathbb{E}^{4}$ is evolute
surface of $M~$if and only if the mean curvature of $M$ is equal to one.
\end{corollary}

In \cite{Ce} M. A. Cheshkova gave the following results.

\begin{theorem}
\cite{Ce} Let $M$ be local surface in $\mathbb{E}^{4}$. If the normal
transport surface $\widetilde{M}$ in $\mathbb{E}^{4}$ is evolute surface of $%
M~$then $M$ has flat normal bundle.
\end{theorem}

\begin{theorem}
\cite{Ce} The minimal surfaces have no evolutes.
\end{theorem}

\begin{example}
Let $M$ is a translation surface $x(u,v)=\alpha (u)+\beta (v)$ in $\mathbb{E}%
^{4}$ , then the translation curves $\alpha (u)=\left( \alpha _{1}(u),\alpha
_{2}(u),0,0\right) $ and $\beta (v)=\left( 0,0,\beta _{1}(v),\beta
_{2}(v)\right) $ are plane curves of mutually orthogonal $2$-planes. The
surface $\widetilde{M}$ is a translation surface, and its translation curves
$\widetilde{\alpha }(u)$, $\widetilde{\beta }(u)$ are the evolutes of the
curves $\alpha (u)$, $\beta (u)$. If $u,v,\kappa _{\alpha },\kappa _{\beta }$
and $\{t_{\alpha },n_{\alpha }\},\{t_{\beta },n_{\beta }\}$ are the arc
length, the curvature, and the Frenet frame of the curves $\alpha (u)$ and $%
\beta (v)$, correspondingly, then
\begin{eqnarray*}
\widetilde{x}(u,v) &=&\alpha (u)+\frac{1}{\kappa _{\alpha }}n_{\alpha
}(u)+\beta (v)+\frac{1}{\kappa _{\beta }}n_{\beta }(v) \\
&=&\alpha (u)+\beta (v)+\frac{1}{\kappa _{\alpha }}n_{\alpha }(u)+\frac{1}{%
\kappa _{\beta }}n_{\beta }(v) \\
&=&x(u,v)+\frac{1}{\kappa _{\alpha }}n_{\alpha }(u)+\frac{1}{\kappa _{\beta }%
}n_{\beta }(v).
\end{eqnarray*}%
The tangent space to $\widetilde{M}$ at an arbitrary point $p=\widetilde{x}%
(u,v)$ of $\widetilde{M}$ is spanned by%
\begin{equation*}
\begin{array}{l}
\vspace{2mm}\widetilde{x}_{u}=\left( \frac{1}{\kappa _{\alpha }}\right)
^{\prime }n_{\alpha }(u), \\
\widetilde{x}_{v}=\left( \frac{1}{\kappa _{\beta }}\right) ^{\prime
}n_{\beta }(v).%
\end{array}%
\end{equation*}%
Consequently, the normal transport surface $\widetilde{M}$ of $M$ satisfies
the equality
\begin{equation*}
\left \langle \widetilde{x}_{u_{i}},x_{u_{j}}\right \rangle =0.
\end{equation*}%
Hence, $\widetilde{M}$ is the evolute of $M$ \cite{Ce}.
\end{example}

\section{\textbf{An Application}}

Rotation surfaces were studied in \cite{Vr} by Vranceanu as surfaces in $%
\mathbb{E}^{4}$ which are defined by the following parametrization
\begin{equation}
M:x(u,v)=(r(v)\cos v\cos u,r(v)\cos v\sin u,r(v)\sin v\cos u,r(v)\sin v\sin
u)  \label{7.1}
\end{equation}%
where $r(v)$ is a real valued non-zero function.

We have the following result.

\begin{theorem}
Let $\widetilde{M}$ be a normal transport surface of the Vranceanu surface $%
M $ given with the parametrization (\ref{5.1}). If $\ \widetilde{M}$ is an
evolute surface of $M$ in $\mathbb{E}^{4}$ then
\begin{equation}
\widetilde{M}:\widetilde{x}(u,v)=\lambda \mu e^{\mu v}(-\sin v\cos u,-\sin
v\sin u,\cos v\cos u,\cos v\sin u),  \label{7.2}
\end{equation}%
where $\lambda $ and $\mu $ are non zero constants.
\end{theorem}

\begin{proof}
Let $M$ be a Vranceanu surfaces given with the parametrization (\ref{7.1}).
We choose a moving frame $\{X_{u},X_{v},N_{1},N_{2}\}$ such that $%
X_{u},X_{v} $ are tangent to $M$ and $N_{1},N_{2}$ normal to $M$ as given
the following (see, \cite{Yo1}):
\begin{eqnarray*}
x_{u} &=&r(-\cos v\sin u,\cos v\cos u,-\sin v\sin u,\sin v\cos u), \\
x_{v} &=&(B(v)\cos u,B(v)\sin u,C(v)\cos u,C(v)\sin u), \\
N_{1} &=&\frac{1}{A}(-C(v)\cos u,-C(v)\sin u,B(v)\cos u,B(v)\sin u), \\
N_{2} &=&(-\sin v\sin u,\sin v\cos u,\cos v\sin u,-\cos v\cos u),
\end{eqnarray*}%
where
\begin{eqnarray*}
A(v) &=&\sqrt{r^{2}(v)+(r^{\prime }(v))^{2}},\text{ }\,B(v)=r^{\prime
}(v)\cos v-r(v)\sin v, \\
C(v) &=&r^{\prime }(v)\sin v+r(v)\cos v.
\end{eqnarray*}

Suppose that $\widetilde{M}$ is the normal transport surface of the
Vranceanu surface $M$ in $\mathbb{E}^{4}$ then we have
\begin{eqnarray}
\left \langle \widetilde{x}_{u},x_{u}\right \rangle &=&r^{2}(v)-f_{1}\left(
\frac{r^{2}(v)}{\sqrt{r^{2}(v)+(r^{\prime }(v))^{2}}}\right) ,  \notag \\
\left \langle \widetilde{x}_{u},x_{v}\right \rangle &=&f_{2}r(v),
\label{7.3} \\
\left \langle \widetilde{x}_{v},x_{u}\right \rangle &=&f_{2}r(v),  \notag \\
\left \langle \widetilde{x}_{v},x_{v}\right \rangle &=&r^{2}(v)+(r^{\prime
}(v))^{2}+f_{1}\left( \frac{r(v)r^{\prime \prime }(v)-2(r^{\prime
}(v))^{2}-r^{2}(v)}{\sqrt{r^{2}(v)+(r^{\prime }(v))^{2}}}\right) .  \notag
\end{eqnarray}

Furthermore, if $\widetilde{M}$ is an evolute surface of the Vranceanu
surface $M$ in $\mathbb{E}^{4}$ then using (\ref{6.1}) with (\ref{7.3}) we
obtain%
\begin{equation}
\begin{array}{l}
\vspace{2mm}f_{2}=0, \\
f_{1}=\sqrt{r^{2}(v)+(r^{\prime }(v))^{2}}.%
\end{array}
\label{7.4}
\end{equation}%
Moreover, from the first and fourth equations of (\ref{7.3}) one can get
\begin{equation*}
r(v)r^{\prime \prime }(v)-(r^{\prime }(v))^{2}=0
\end{equation*}%
which has a non-trivial solution
\begin{equation}
r(v)=\lambda e^{\mu v}  \label{7.5}
\end{equation}

As a consequence of (\ref{7.4}) with (\ref{7.5}) we get the desired result.
\end{proof}

\begin{remark}
The Vranceanu surface given with $r(v)=\lambda e^{\mu v}$ is a flat surface
with vanishing normal curvature \cite{ABBKMO}.
\end{remark}

\vskip3mm
\begin{tabular}{l}
Kadri Arslan \& Bet\"{u}l Bulca \\
Department of Mathematics \\
Uluda\u{g} University \\
16059 Bursa, TURKEY \\
E-mails: arslan@uludag.edu.tr,\ \  \\
bbulca@uludag.edu.tr\ \ \ \ \ \ \ \ \
\end{tabular}

\vskip3mm
\begin{tabular}{l}
Beng\"{u} K\i l\i \c{c} Bayram \\
Department of Mathematics \\
Bal\i kesir University \\
Bal\i kesir, TURKEY \\
E-mail: benguk@bal\i kesir.edu.tr,%
\end{tabular}

\vskip3mm
\begin{tabular}{l}
G\"{u}nay \"{O}zt\"{u}rk \\
Department of Mathematics \\
Kocaeli University \\
Kocaeli, TURKEY \\
E-mail: ogunay@kocaeli.edu.tr,%
\end{tabular}

\end{document}